\documentclass[12pt]{amsart}

\usepackage{amssymb}
\usepackage{amsthm}
\usepackage{mathrsfs}
\usepackage[sort]{cite}

\theoremstyle{plain}
\newtheorem{theorem}[equation]{Theorem}
\newtheorem{lemma}[equation]{Lemma}
\newtheorem{proposition}[equation]{Proposition}
\newtheorem*{Theorem}{Theorem}
\newtheorem*{WACZ}{Weighted Atiyah Conjecture}
\newtheorem*{RAWACZ}{Right-Angled Weighted Atiyah Conjecture}

\theoremstyle{remark}
\newtheorem{remark}[equation]{Remark}
\newtheorem*{Remark}{Remark}

\numberwithin{equation}{section}

\newcommand{\tr}{\operatorname{tr}}
\newcommand{\diag}{\operatorname{diag}}
\newcommand{\pr}{\operatorname{pr}}

\newcommand{\<}{\langle}
\renewcommand{\>}{\rangle}

\newcommand{\bbR}{\mathbb R}
\newcommand{\bbZ}{\mathbb Z}
\newcommand{\bbQ}{\mathbb Q}
\newcommand{\bbC}{\mathbb C}

\newcommand{\LG}{L^2_{\bq}G}
\newcommand{\LW}{L^2_{\bq}W}
\renewcommand{\NG}{{\mathcal N}_{\bq}G}
\newcommand{\NW}{{\mathcal N}_{\bq}W}
\newcommand{\cN}{\mathcal N}
\newcommand{\cS}{\mathcal S}
\newcommand{\cH}{\mathcal H}
\newcommand{\cA}{\mathcal A}

\renewcommand{\bar}[1]{\overline{#1}}
\renewcommand{\Re}{\operatorname{Re}}

\newcommand{\bq}{\mathbf q}
\newcommand{\ttau}{\tilde{\tau}}
\newcommand{\bv}{\nu} 
\newcommand{\bm}{\mathbf m}
\newcommand{\bn}{\mathbf n}
\newcommand{\bk}{\kappa}
\newcommand{\tk}{\tilde{\kappa}}
\newcommand{\tu}{\tilde{\kappa}_{\emptyset}}

\title[The Atiyah conjecture for the infinite dihedral group]
{The Atiyah conjecture for the Hecke algebra of
  the infinite dihedral group}

\author[B.~Okun]{Boris Okun}
\address{Department of Mathematical Sciences\\
University of Wisconsin\\
Milwaukee, WI 53201}
\email{okun@uwm.edu}

\author[R.~Scott]{Richard Scott}
\thanks{The second author was supported in part by an IBM Research Grant awarded by Santa Clara University}
\address{Department of Mathematics and Computer Science\\
Santa Clara University\\
Santa Clara, CA  95053}
\email{rscott@scu.edu}

\begin{document}

\begin{abstract}
We prove a generalized version of the Strong Atiyah Conjecture for the
infinite dihedral group $W$, replacing the group von Neumann
algebra $\cN W$ with the Hecke--von Neumann algebra $\NW$.
\end{abstract}

\maketitle

\section{Introduction}

Let $W$ be a discrete group, let $\bbR W$ denote its group algebra
over $\bbR$, and let $L^2 W$ denote the Hilbert space completion of
$\bbR W$ with respect to the standard inner product.  Let $\cN W$ be
the von Neumann algebra obtained by taking the bounded operators on $L^2 W$
that commute with the right $\bbR W$-action.  We regard $\cN W$ as an
algebra of (left) operators on $L^2W$.  Then any closed $\bbR
W$-invariant subspace  $V\subseteq (L^2 W)^n$ has a well-defined
von Neumann dimension, which we denote by $\dim_{W}V$.  Examples of
such subspaces arise naturally in $L^2$-homology calculations as
kernels and image closures of equivariant boundary maps and
laplacians, all of which can be represented as right-multiplication by
matrices with entries in $\bbQ W$, the rational group ring.  
The Atiyah Conjecture asserts that any invariant subspace of the
form $\ker R_M$ where $R_M:(L^2W)^n\rightarrow(L^2W)^m$ is right
multiplication by a matrix $M$ with entries in $\bbQ W$ will have
rational von Neumann dimension.  In full generality this conjecture is
false; a counterexample was first given by Austin 
\cite{A}, see also \cite{Gr, PSZ}.  In all of these counterexamples
the group has finite subgroups of arbitrarily large order.  For groups
with bounded torsion, a stronger form of the conjecture, which
specifies denominators of these rational dimensions, is still open. 
Namely, if $\Lambda$ denotes the additive subgroup of $\bbR$ generated by $\{1/|H|\}$ where
$H$ ranges over finite subgroups of $W$, then the Strong Atiyah
Conjecture asserts that $\dim_W\ker R_M\in\Lambda$.  

In the case where $W$ is a right-angled Coxeter group $W$, the
Strong Atiyah Conjecture was recently settled by Linnell, Okun
and Schick \cite{LOS}. It remains open for arbitrary Coxeter groups.
Here we consider a version of Atiyah's conjecture that makes sense for
Hecke algebras.  We let $W$ be a Coxeter group 
with standard generating set $S$, and let $\bbR_{\bq} W$ denote the
Hecke algebra corresponding to $W$ with real deformation multiparameter
$\bq=(q_s)_{s\in S}$ (as usual, we require $q_s=q_t$ whenever $s$ and
$t$ are conjugate in $W$).  This algebra has a canonical $\bbR$-basis 
$\{T_w\;|\; w\in W\}$, and multiplication determined by 
\[T_sT_w=\left\{\begin{array}{ll}
T_{sw}&\mbox{if $|sw|>|w|$}\\ 
(q_s-1)T_{w}+q_sT_{sw} &\mbox{if $|sw|<|w|$}\end{array}\right.\]
for all $s\in S$, and $w\in W$.  We let $q^w$ denote the product
$q_{s_1}\cdots q_{s_n}$ where $s_1\cdots s_n$ is a reduced expression
for $w$.  It follows from Tits' solution to the word problem for $W$
that $q^w$ is independent of the choice of reduced expression.  The
algebra $\bbR_{\bq}W$ can be regarded as a deformation of the group
algebra $\bbR W$, and the canonical inner product on $\bbR W$ deforms
to the inner product on $\bbR_{\bq} W$ defined by
$\<T_w,T_{w'}\>=q^w\delta_{w,w'}$ for all $w,w'\in W$. In particular,
the basis elements $T_w$ are orthogonal, and left and right
multiplication by $T_s$  (for $s\in S$) are self-adjoint operators. 
We let $\LW$ denote the Hilbert space completion with respect to
this inner product.  Again one obtains a von Neumann algebra, which we
denote by $\NW$, by taking the bounded operators on $\LW$ that commute
with the right $\bbR_{\bq}W$-action.  And again one obtains von
Neumann dimensions for closed $\bbR_{\bq} W$-invariant subspaces
$V\subseteq(L^2_{\bq} W)^n$.  We denote this dimension by
$\dim_W^{\bq}V$.  

To motivate the algebraic formulation of the Atiyah Conjecture in the
context of Hecke algebras, we recall some properties of Coxeter
groups and reflection actions (a good reference for this material is
\cite{Davis-book}[Chapter 20]).  A subet $J\subseteq S$ is called {\em
  spherical} if the parabolic subgroup $W_J$ generated by $J$ is finite, and we
let $\cS$ denote the set of spherical subsets of $S$.  For any
$J\in\cS$, we let $W_J({\bq})$ denote the growth series (a polynomial
in this case) of $W_J$ defined by  
\[W_J({\bq})=\sum_{w\in W_J} q^w,\]
and we let $a_J$ be the element of $\bbR_{\bq} W$ defined by 
\[a_J=\frac{1}{W(\bq)}\sum_{w\in W_J} T_w.\]
Right-multiplication by $a_J$ defines an orthogonal projection
from $\LW$ onto a closed (left) $\bbR_{\bq} W$-invariant subspace,
which we denote by $A_J$.  The von-Neumann dimension of this subspace is 
\[\dim_W^{\bq}A_J=\frac{1}{W_J({\bq})}.\]

Given a reflection action of $W$ on a CW-complex $X$, there is a
corresponding cochain complex of $\NW$-modules, and the
``weighted'' $L^2_{\bq}$-Betti numbers of $X$ are defined as the
von-Neumann dimensions of the corresponding cohomology groups.  In
\cite{DDJO}[Section 7] it is proved that these Betti numbers are
continuous with respect to the multiparameter $\bq$ and, in light of
Atiyah's question, the authors ask whether or not these Betti numbers
are piecewise rational functions.  A purely algebraic version of the question can
be obtained by first noting that the $\NW$-modules in the weighted chain
complex all decompose into orthogonal direct sums of $A_J$'s, and the
boundary and coboundary maps can all be represented by matrices whose  
entries are $\bbZ$-linear combinations of the ${a_J}'s$. 
\footnote{In \cite{Davis-book}, \cite{DDJO}, and  \cite{Dymara}, the
  boundary map formula has coefficients involving the parameters $\bq$ and square roots.
  However, if one scales the $L^2$ norms of the cells in 
  each orbit appropriately, and expresses the boundary map in terms of
  the projection operators $a_J$, the coefficients all become integers.
  The weighted Betti numbers remain unchanged by this scaling.}

To get an algebraic formulation of the conjecture, we replace boundary
and coboundary maps with a suitable class of matrices, and ask about
von-Neumann dimensions of the kernels.  To have a
canonical specialization of each matrix for different values of the
multiparameter $\bq$, we let $\bbQ(\bq)$ denote the formal ring of
rational functions in the indeterminates $q_s$, $s\in S$, and we
define $\cH W$ to be the abstract Hecke algebra over $\bbQ(\bq)$ with
generators  
$T_w$, $w\in W$, and the same multiplication rules given above for
$\bbR_{\bq}W$.  (To avoid extra notation, we
use the same symbols $\{q_s\}$ both for formal indeterminates and for
  real parameters.)  By allowing polynomial denominators, all of the
projections $a_J$ are well-defined elements of $\cH W$, and we let
$\cA W$ denote the subalgebra they generate.  Since denominators in
$\cA W$ will always be polynomials with non-negative coefficients,
there will be no division by zero problems when specializing to any
multiparameter $\bq\in(\bbR_{>0})^S$.  

\begin{WACZ}
Let $M$ be an $n\times m$ matrix with entries in $\cA W$ and for any
multiparameter $\bq\in(\bbR_{>0})^S$, let $M_{\bq}$ denote the
specialization of this matrix to $\bbR_{\bq}W$.  Then the von Neumann 
dimension of the kernel of right multiplication by $M_{\bq}$ on
$(L^2_{\bq}W)^n$ is a piecewise rational function of 
the form \[\dim_W^{\bq}\ker
R_{M_{\bq}}=\sum_{J\in\cS}\frac{n_J({\bq})}{W_J({\bq})}\]  
where the numerators $n_J$ are piecewise-constant integer functions of
$\bq$.
\end{WACZ}    

One complication in trying to establish this conjecture is that,
in general, subgroups of $W$ do not correspond to subalgebras of
$\bbR_{\bq} W$.  If $W$ is right-angled, however, there is a canonical
isomorphism between $\bbR_{\bq}W$ and the ordinary group algebra $\bbR
W$ (see \cite{OS} and Section~\ref{s:RAHvN}, below).  Thus, for any
subgroup $G\subseteq W$, there is a canonical subalgebra $\bbR_{\bq}
G\subseteq \bbR_{\bq} W$ isomorphic to the group subalgebra $\bbR
G\subseteq\bbR W$.  Moreover, because this isomorphism is induced by
identifying the idempotents $a_J$ in $\bbR W$ with those in
$\bbR_{\bq}W$, the statement of the Weighted Atiyah
Conjecture in the right-angled setting takes a slightly simpler form 
(which we give at the end of Section~\ref{s:RAHvN}).

The point of this paper is to establish the conjecture for the first
nontrivial example in the right-angled setting, namely, when $W$ is
the infinite dihedral group.  Although the result is admittedly
limited in scope, the proof is surprisingly subtle and much more
involved than the corresponding result in the Coxeter group setting.
In what follows, we assume $W$ is the infinite dihedral group with
generators $s$ and $t$, and we let $G$ be the infinite cyclic subgroup
of index $2$ generated by $st$. The proof of the (non-weighted) Atiyah
Conjecture for $W$ boils down to two facts.  First, if
$V\subseteq(L^2W)^n$ is a left $\bbR W$-invariant closed subspace then
$\dim_GV=2\dim_WV$.  This follows from the orthogonal decomposition          
\[L^2 W=L^2 G\oplus(L^2G)s\cong(L^2 G)^2.\]
And second, (right) multiplication in $L^2 G$ by a nonzero element
of the group algebra $\bbR G$ has trivial kernel. This follows from a
Fourier series argument.  When $q_s\neq 1$ or $q_t\neq 1$, the
argument breaks down in two places: first, $\LG$ and
$(\LG)s$ are not orthogonal, and second, $\LG$ has
nontrivial submodules of the form $\ker R_M$.  We address these
difficulties by describing a finer orthogonal decomposition of
$\LW$.  We then prove the following case of the Weighted Atiyah
Conjecture.     

\begin{Theorem}
Let $W$ be the infinite dihedral group $\<s,t\:|\; s^2=t^2=1\>$, and
let $M$ be a matrix with entries in $\cA W$.  Then for any
multiparameter $\bq=(q_s,q_t)$, we have 
\[\dim_W^{\bq}\ker
R_{M_{\bq}}=n_{\emptyset}+\frac{n_s}{1+q_s}+\frac{n_t}{1+q_t}\]
where $n_{\emptyset},n_s,n_t$ are piecewise constant integer functions of $\bq$. 
\end{Theorem}

To make the paper easier to follow, we outline here the key steps in
the proof of the main theorem.  The first step is to identify $\bbR_q
W$ with $\bbR W$ using the canonical isomorphism and then to pass to
the subalgebra $\bbR G$ where $G$ is the free abelian subgroup of $W$
generated by the translation $st$.  The advantage of $\bbR G$ over
$\bbR W$ is that the former is isomorphic to the commutative ring of
Laurent polynomials, and 
matrices over this ring are easier to work with.  We then consider the
action of the group generator $st$ on $\LG$, 
letting $K_+$ and $K_-$ denote the $+1$ and $-1$-eigenspaces,
respectively.  We obtain an orthogonal decomposition  
\[\LG=K_+\oplus K_-\oplus K_{\emptyset}\]
where $K_{\emptyset}$ is the orthogonal complement of $K_+$ and $K_-$.  
We then show that right multiplication by any element $y\in\bbR G$, 
restricted to any of these three summands, is either an isomorphism or
the zero map (Proposition~\ref{prop:Virred}).  This follows from two
facts.  First, being a Laurent poynomial in one variable, $y$ factors
into linear factors over $\bbC$.  Second, $+1$ and $-1$ are the only
complex eigenvalues for the action of  $st$ on $\LG$.
Section~\ref{s:GonLW} is devoted entirely to this second fact, which is
the main technical result of the paper. 

We then extend this decomposition to $\LW$, proving that 
\begin{align}\label{al:dec}
\LW & = K_+\oplus K_-\oplus K_{\emptyset}\oplus K_{\emptyset}s
\end{align}
as $\NG$-modules (Proposition~\ref{prop:W-decomp}).

\begin{Remark}
For any Coxeter group $W$, Davis et al. \cite[Theorem~9.11]{DDJO}
prove a decomposition 
theorem for $\LW$ that generalizes the
decomposition of Solomon \cite{So} for finite Coxeter groups (and the
ordinary group algebra).  In the case of
the infinite dihedral group, the two subspaces $K_+$ and $K_-$ in our
decomposition are not just $\NG$-modules, but they are also
$\NW$-modules, and can be used to give an even finer decomposition of
$\LW$ than that in \cite{DDJO}. The subspace  $K_+$ corresponds to
either the constant functions or ``harmonic'' functions (denoted by
$A^S$ or $H^S$, respectively, in \cite{DDJO}), but the invariant subspace $K_-$ is
new. It can be regarded as the image of $K_+$ under one of the
``partial $j$'' automorphisms described in \cite[Section~9]{OS} and is
a proper invariant subspace of one of the summands in the decomposition of Davis et al.  
\end{Remark}

Given an $\bbR W$-invariant subspace $V\subseteq(\LW)^n$, we obtain a
corresponding decomposition  
\[V=V_+\oplus V_-\oplus V_{\emptyset}\]
where $V_+\subseteq (K_+)^n$, $V_-\subseteq (K_-)^n$, and
$V_{\emptyset}\subseteq(K_{\emptyset}\oplus K_{\emptyset}s)^n$ (Proposition~\ref{prop:V-W-decomp}).  We
then prove that if $V$ is the kernel of an $\bbR W$-matrix, then as
$\NG$-modules we have isomorphisms, 
\[V_+\cong (K_+)^a,\;\; V_-\cong (K_-)^b,\;\; V_{\emptyset}\cong
(K_{\emptyset})^{c}\]
where $a,b,c$ are nonnegative integers (Lemmas~\ref{lem:G-main} and \ref{lem:kerW2G}).  The proof of this requires
one to first show that right multiplication by an $\bbR W$-matrix
corresponds to right multiplication by an $\bbR G$-matrix with respect
to the decomposition (\ref{al:dec}) above, and then to use the fact
that matrices over Laurent polynomial rings are essentially
diagonalizable.   This means that right multiplication by an $\bbR
G$-matrix on any of the subspaces $(K_+)^n$, $(K_-)^n$, or
$(K_{\emptyset}\oplus K_{\emptyset}s)^n\cong (K_{\emptyset})^{2n}$ reduces to the $1$-dimensional case, where (by Proposition~\ref{prop:Virred}, mentioned above), the kernel is either trivial or the entire space.    

Finally, we calculate the $\NG$-dimensions of the
modules $V_+\cong (K_+)^a$, $V_-\cong (K_-)^b$, and $V_{\emptyset}\cong
(K_{\emptyset})^{c}$ (Lemma~\ref{lem:Gdims}), relate these to their
$\NW$-dimensions (Lemma~\ref{lem:dimG2W}), and then complete the proof (Theorem~\ref{thm:final}).     

\section{Hecke--von Neumann algebras for right-angled Coxeter groups}\label{s:RAHvN}

Let $W$ be a right-angled Coxeter group with generating set $S$, and
let $\bq=(q_s)_{s\in S}$ be a real-valued $S$-tuple satisfying $q_s>0$
for all $s\in S$.  We let $\bbR_{\bq}W$ denote the corresponding Hecke
algebra and note that in addition to the multiplication
formulas from the introduction 
\[T_sT_w=\left\{\begin{array}{ll}
T_{sw}&\mbox{if $|sw|>|w|$}\\ 
(q_s-1)T_{w}+q_sT_{sw} &\mbox{if $|sw|<|w|$}\end{array}\right.,\]
there are analogous right-multiplication formulas
\[T_wT_s=\left\{\begin{array}{ll}
T_{ws}&\mbox{if $|ws|>|w|$}\\ 
(q_s-1)T_{w}+q_sT_{ws} &\mbox{if $|ws|<|w|$}\end{array}\right..\]

In a previous paper, the authors noted that for right-angled Coxeter
groups, there is a canonical isomorphism $\phi:\bbR
W\rightarrow\bbR_{\bq} W$ of $\bbR$-algebras induced by  
\[\phi(s)=\frac{1-q_s}{1+q_s}+\frac{2}{1+q_s}T_s\]
for all $s\in S$ (see \cite{OS}[Corollary~9.7]).  This isomorphism is
induced by mapping each of the idempotents $a_s=\frac{1+s}{2}$ in
$\bbR W$ to the corresponding idempotent
$a_s=\frac{1+T_s}{1+q_s}\in\bbR_{\bq}W$.  In fact, (and this is unique
to the right-angled setting) for any spherical subset $J\in\cS$, one has 
\[\phi(a_J)=a_J.\]

The Hecke algebra $\bbR_{\bq} W$ has an $\bbR$-basis $\{T_w\}$
canonically indexed by elements of $W$: each $T_w$ is a product
$T_w=T_{s_1}\cdots T_{s_n}$ where $s_1\cdots s_n$ is a reduced
expression for $w$.  We let 
$\tau_w=\phi^{-1}(T_w)$, keeping in mind that $\tau_w$ depends on the
choice of $\bq$.  We then have two bases $\{w\;|\;w\in W\}$ and
$\{\tau_w\;|\;w\in W\}$ for the group algebra $\bbR W$ (which coincide
if and only if $q_s=1$ for all $s\in S$).  Throughout the paper, we
shall denote the unit element $\tau_1=\phi^{-1}(T_1)$ by $1$ and
identify $\bbR$ with the constants $\bbR\tau_1\subseteq\bbR
W$.  From the definition of $\phi$ we have, for all $s\in S$,  
\begin{align}\label{al:s-to-taus}
s=\frac{1-q_s}{1+q_s}+\frac{2}{1+q_s}\tau_s,
\end{align}
and since $\phi$ is an algebra isomorphism, the multiplication
formulas for the Hecke basis $T_w$ correspond to the same formulas
for the $\tau_w$ basis in the group algebra, namely
\begin{equation}\label{eq:tau-mult}
\tau_s\tau_w=\left\{\begin{array}{ll}
\tau_{sw}&\mbox{if $|sw|>|w|$}\\ 
(q_s-1)\tau_{w}+q_s\tau_{sw} &\mbox{if $|sw|<|w|$}\end{array}\right.
\end{equation}
and 
\begin{equation}
\tau_w\tau_s=\left\{\begin{array}{ll}
\tau_{ws}&\mbox{if $|ws|>|w|$}\\ 
(q_s-1)\tau_{w}+q_s\tau_{ws} &\mbox{if $|ws|<|w|$}\end{array}\right..
\end{equation}

Pulling back the inner product on $\bbR_{\bq} W$ from the introduction, we
obtain, a corresponding inner product $\<,\>_{\bq}$ on the group algebra
$\bbR W$.  This inner product is given by   
\[\<\tau_w,\tau_{w'}\>_{\bq}=\<T_w,T_{w'}\>=q^w\delta_{w,w'}\]
for all $w,w'\in W$.  

We then identify the Hilbert space completion $L^2_{\bq} W$
with the completion of the group algebra $\bbR W$ with 
respect to the inner product $\<,\>_{\bq}$.  As in
\cite[Section~19.2]{Davis-book}, one obtains a von Neumann algebra
$\NW$ of (left) operators on $L^2_{\bq} W$ by  
taking all bounded operators that commute with the right $\bbR
W$-action.  Alternatively, we say that an element $x\in\LW$ is {\em
bounded} if there is some constant $C$ such that $\|xy\|\leq C\|y\|$
for all $y\in\bbR W$.  The von Neumann algebra $\NW$ can then be
identified with the weak closure of the subset of $\LW$ consisting of
bounded elements acting on the left of $\bbR W$.  (Similarly, there is
a von Neumann algebra of right operators on $\LW$, which we
also denote by $\NW$.  The context will usually determine which
algebra we are using.)  

A basic fact we shall need about the inner product 
$\<,\>_{\bq}$ on $L^2_{\bq} W$ is that for any generator $s\in S$, left and
right multiplication by $s$ and $\tau_s$ are self-adjoint.   

\begin{proposition}\label{prop:gens-adjoint} For any $s\in S$ and $x,y\in\LW$,
\[\<sx,y\>_{\bq}=\<x,sy\>_{\bq} \;\;\mbox{and}\;\;\<xs,y\>_{\bq}=\<x,ys\>_{\bq}\]
and 
\[\<\tau_s x,y\>_{\bq}=\<x,\tau_s y\>_{\bq}
\;\;\mbox{and}\;\;\<x\tau_s,y\>_{\bq}=\<x,y\tau_s\>_{\bq}.\]
\end{proposition}

\begin{proof}{}
In \cite[Proposition~2.1]{Dymara}, any Hecke algebra $\bbR_{\bq}W$,
together with the involution $\ast$ defined by $T_w^*=T_{w^{-1}}$ and
the inner product defined by $\<T_w,T_{w'}\>=q^w\delta_{w,w'}$, is shown
to satisfy the axioms for a Hilbert algebra structure in the sense of
Dixmier \cite{Dix}.  In particular, for any $x\in\bbR_{\bq}W$, left
(respectively, right) multiplication by $x^{\ast}$ is the adjoint of
left (resp., right) multiplication by $x$ with respect to $\<,\>$.  When
$W$ is right-angled, the isomorphism $\phi^{-1}:\bbR_{\bq}W\rightarrow
\bbR W$ induces a Hilbert algebra structure on $\bbR W$ where the
inner product is $\<,\>_{\bq}$ and the $\ast$-involution is given by
$w^{\ast}=w^{-1}$ on the $\{w\}$ basis and
$\tau_w^{\ast}=\tau_{w^{-1}}$ on the $\{\tau_w\}$ basis.  Thus,
$s^{\ast}=s$, and $\tau_s^{\ast}=\tau_s$ for all $s\in S$. 
\end{proof}

For any positive integer $n$, we let $(L^2_{\bq} W)^n$ denote the Hilbert space direct sum of $n$ copies of
$L^2_{\bq} W$, and we let $\epsilon_1,\ldots,\epsilon_n$ denote the
standard basis; in other words $\epsilon_i=(0,\ldots,0,1,0,\ldots,0)$ where the
$1$ in the $i$th position represents the element $1\in\bbR W$.
Any closed (left) $\bbR W$-invariant subspace  $V\subseteq (L^2_{\bq} W)^n$
will be called a {\em Hilbert $\NW$-module}, and has {\em von Neumann
dimension} defined by   
\[\dim_W^{\bq}V=\sum_{i=1}^{n}\<\pr_V(\epsilon_i),\epsilon_i\>_{\bq}\]
where $\pr_V:(L^2_{\bq} W)^n\rightarrow V$ is orthogonal projection
onto $V$.  An {\em isomorphism} of Hilbert modules is an $\bbR
W$-equivariant Hilbert space isomorphism. Isomorphic Hilbert modules
have the same von Neumann dimension (see e.g.,
\cite[Theorem~1.12]{Lueck}).  Similarly, if $G$ is any subgroup of
$W$, we can restrict the 
inner product $\<,\>_{\bq}$ to $\bbR G$.  The Hilbert space completion
$L^2_{\bq} G$ can then be identified with the closure of $\bbR G$ in
$L^2_{\bq} W$.  As above, one defines the von Neumann algebra $\NG$
to be the algebra of bounded operators on $L^2_{\bq} G$ that commute
with the right $\bbR G$-action.  A Hilbert $\NG$-module $V$ is defined
by replacing $W$ with $G$ in the previous paragraph, and its von
Neumann dimension will be denoted by $\dim_G^{\bq}V$.

With this identification of $\LW$ (for {\em any} $\bq$) with a
suitable completion of the ordinary group algebra $\bbR W$, the
statement of the Weighted Atiyah Conjecture is simplified.  In
particular, the specialization homomorphism $\cA
W\rightarrow\bbR_{\bq}W$, when composed with the isomorphism
$\phi^{-1}:\bbR_{\bq} W\rightarrow\bbR W$ is independent of $\bq$.
This means that {\em for all $\bq$}, we can regard $M$ as a matrix
with entries in the rational group algebra $\bbQ W$.  Clearing
denominators, we obtain the following. 

\begin{RAWACZ}
Let $W$ be a right-angled Coxeter group, and let $M$ be an $n\times m$
matrix with entries in the integer group ring $\bbZ W$.  Then 
\[\dim_W^{\bq}\ker R_M=\sum_{J\in\cS}\frac{n_J({\bq})}{W_J({\bq})}\]  
where the numerators $n_J$ are piecewise-constant integer functions of
$\bq$.
\end{RAWACZ}

\section{The $G$-action on $\LW$}\label{s:GonLW}

For the remainder of the paper $W$ will be the infinite dihedral
group with standard generators $s$ and $t$.  We let $G$ be the
infinite cyclic subgroup generated by the product $st$, and we
consider the operator on $L^2_{\bq} W$ defined by right
multiplication by $st$.  We shall prove that the only possible
eigenvalues for this operator are $1$ and $-1$ (and even these may or
may not occur depending on the values of the parameters $q_s$ and
$q_t$).  The same 
result holds for left multiplication by $st$, as well, with
the same resulting eigenvalues and eigenvectors, but we shall omit the
argument since it is virtually identical to that for right-multiplication.

We work both with the orthogonal basis $\{\tau_w\}$ for
$\LW$ and the orthonormal basis $\{\ttau_w\}$ defined by  
\[\ttau_w=(1/\sqrt{q^w})\tau_w.\]
For the $\bbR W$-action on $\LW$, we introduce the special
elements $a_s$ and $a_t$ defined by 
\begin{align}\label{al:as}
a_s:=\frac{1+s}{2}=\frac{1+\tau_s}{1+q_s}\;\;\mbox{and}\;\; a_t:=
\frac{1+t}{2}=\frac{1+\tau_t}{1+q_t}
\end{align}
(the equations follow from (\ref{al:s-to-taus})).

One checks easily using the fact that $s^2=1$ and $t^2=1$ that $a_s$ and
$a_t$ are self-adjoint idempotents, as are their complements
$h_s=1-a_s$ and $h_t=1-a_t$.  The latter are 
given in terms of the bases $\{w\}$ and $\{\tau_w\}$ by 
\begin{align}\label{al:hs}
h_s=\frac{1-s}{2}=\frac{q_s-\tau_s}{1+q_s}\;\;\mbox{and}\;\; h_t=
\frac{1-t}{2}=\frac{q_t-\tau_t}{1+q_t}.
\end{align}
Our first step is to replace the operator $st$ with $a_s-a_t$.  

\begin{lemma}\label{lem:lambda2mu}
The vector $\bv\in\LW$ is an eigenvector for $st$ with
eigenvalue $\lambda$ if and only if $\bv$ is an eigenvector for
$a_s-a_t$ with eigenvalue 
\[\mu=\pm\sqrt{\frac{1}{2}-\frac{1}{2}\Re\lambda}.\]
\end{lemma}

\begin{proof}{}
Let $\bv$ be an eigenvector for $st$ with eigenvalue $\lambda$.
Since $s$ and $t$ are self-adjoint involutions, $st$ is a unitary
operator with $(st)^*=ts$.  It follows that $|\lambda|=1$.  Moreover,
$\bv$ will be in the kernel of the operator 
\[(st-\lambda)(st-\bar{\lambda})=(st)^2+1-2st\Re\lambda=(st+ts-2\Re\lambda)(st).\]
Since $st$ is invertible, $\bv$ will therefore be an eigenvector for
$st+ts$ with eigenvalue $2\Re(\lambda)$.  Using the definition of
$a_s$ and $a_t$ in (\ref{al:as}), we have $s=2a_s-1$ and $t=2a_t-1$, hence
\[st+ts=4(a_sa_t+a_ta_s)-4(a_s+a_t)+1=1-4(a_s-a_t)^2,\]
where the last expression follows from $a_s^2=a_s$ and
$a_t^2=a_t$.  It follows that $\bv$ is an eigenvector for $a_s-a_t$
with eigenvalue $\pm\sqrt{\frac{1}{2}-\frac{1}{2}\Re\lambda}$.  
Tracing the argument backward gives the reverse implication.
\end{proof} 

Next we compute the action of $a_s-a_t$ on the basis vectors
$\{\tau_w\}$.  To avoid denominators, we let $c=(1+q_s)(1+q_t)$ and
let $R$ be the operator 
\[R=c(a_s-a_t)=(q_t-q_s)+(1+q_t)\tau_s-(1+q_s)\tau_t.\]
Any eigenvector of $a_s-a_t$ with eigenvalue $\mu$ will then be a
nonzero vector in the kernel of $R-c\mu$.  We compute the products 
$\tau_w(R-c\mu)$ using the formulas for
right-multiplication by $\tau_s$ and $\tau_t$: 
\begin{align*}
\tau_1(R-c\mu) & = (q_t-q_s-c\mu)\tau_1+(1+q_t)\tau_s-(1+q_s)\tau_t
\end{align*}
and (for $|ws|>|w|$)
\begin{eqnarray*}
\nonumber \tau_{ws}(R-c\mu) &
=&(q_t-q_s-c\mu)\tau_{ws}+\\
&& (1+q_t)[(q_s-1)\tau_{ws}+q_s\tau_w]-(1+q_s)\tau_{wst}\\ 
&=&-(1+q_s)\tau_{wst}+(q_sq_t-1-c\mu)\tau_{ws}+q_s(1+q_t)\tau_w
\end{eqnarray*}
and (for $|wt|>|w|$)
\begin{eqnarray*} 
\nonumber \tau_{wt}(R-c\mu) &
= &(q_t-q_s-c\mu)\tau_{ws}+\\
&&(1+q_t)\tau_{wts}-(1+q_s)[(q_t-1)\tau_{wt}+q_t\tau_w]\\
&=&(1+q_t)\tau_{wts}-(q_sq_t-1+c\mu)\tau_{wt}-q_t(1+q_s)\tau_w.
\end{eqnarray*}

Using the substitutions $\tau_w=\sqrt{q^w}\ttau_w$, we obtain 
formulas with respect to the orthonormal basis:
\begin{align}\label{al:initR}
\ttau_1(R-c\mu) & = \sqrt{q_s}(1+q_t)\ttau_s-\sqrt{q_t}(1+q_s)\ttau_t+(q_t-q_s-c\mu)\ttau_1
\end{align}
and (for $|ws|>|w|$)
\begin{align}\label{al:sR} 
\ttau_{ws}(R-c\mu) &=-\sqrt{q_t}(1+q_s)\ttau_{wst}+(q_sq_t-1-c\mu)\ttau_{ws}+\sqrt{q_s}(1+q_t)\ttau_w
\end{align}
and (for $|wt|>|w|$)
\begin{align} \label{al:tR} 
\ttau_{wt}(R-c\mu) &= \sqrt{q_s}(1+q_t)\ttau_{wts}-(q_sq_t-1+c\mu)\ttau_{wt}-\sqrt{q_t}(1+q_s)\ttau_w.
\end{align}

Now suppose $\bv$ is an eigenvector for $st$ with eigenvalue
$\lambda$ (hence an eigenvector for $R$ with eigenvalue
$c\mu$).  For each $w\in W$, let $\{x_w\}$ be the
coordinates of $\bv$ with respect to the orthonormal basis $\{\ttau_w\}$, i.e.,
$x_w=\<\bv,\ttau_w\>_{\bq}$.  We then have 
\[\bv=\sum_{w\in W}x_w\ttau_w,\]
and $\bv\in L^2_{\bq} W$ if and only if $\sum_w|x_w|^2<\infty$.  

Rewriting the equation $\bv(R-c\mu)=0$ in terms of the coordinates
$\{x_w\}$ using (\ref{al:initR}), (\ref{al:sR}), (\ref{al:tR}), we
obtain the equations 
\begin{align*}
(q_t-q_s-c\mu)x_1+\sqrt{q_s}(1+q_t)x_s-\sqrt{q_t}(1+q_s)x_t=0,
\end{align*}
and (for $|ws|>|w|$)
\begin{align*}
\sqrt{q_s}(1+q_t)x_w+(q_sq_t-1-c\mu)x_{ws}-\sqrt{q_t}(1+q_s)x_{wst}=0,
\end{align*}
and (for $|wt|>|w|$)
\begin{align*}
-\sqrt{q_t}(1+q_s)x_w-(q_sq_t-1+c\mu)x_{wt}+\sqrt{q_s}(1+q_t)x_{wts}=0.
\end{align*}
With the substitutions 
\begin{alignat}{3}
\nonumber \alpha_s & = \sqrt{q_s}+\frac{1}{\sqrt{q_s}} &&
\delta &=&  \frac{\alpha_s}{\alpha_t}\\
\label{al:substitutions}\alpha_t &= \sqrt{q_t}+\frac{1}{\sqrt{q_t}} &\hspace{.3in}\;\;\;\mbox{and}\;\;\;\hspace{.3in}&
\beta &=& \frac{\alpha_{st}}{\alpha_s}-\alpha_t\mu\\
\nonumber \alpha_{st} &= \sqrt{q_sq_t}-\frac{1}{\sqrt{q_sq_t}} &&
\gamma &=& \frac{\alpha_{st}}{\alpha_t}+\alpha_s\mu
\end{alignat}
these three equations simplify to  
\begin{align}\label{al:xinit}
\frac{x_s}{\alpha_s}-\frac{x_t}{\alpha_t}=\left(\mu-\frac{1}{1+q_s}+\frac{1}{1+q_t}\right)x_1,
\end{align}
and (for $|ws|>|w|$)
\begin{align}\label{al:xst}
x_{wst}=\delta^{-1}x_w+\beta x_{ws},
\end{align}
and (for $|wt|>|w|$)
\begin{align}\label{al:xts}
x_{wts}=\delta x_w+\gamma x_{wt}.
\end{align}
Applying these last two formulas consecutively to $x_{wsts}$ we have
\begin{align}\label{al:xsts}
x_{wsts}=\gamma\delta^{-1}x_w+(\delta+\beta\gamma)x_{ws},
\end{align}
and applying them to $x_{wtst}$, we have 
\begin{align}\label{al:xtst}
x_{wtst}=\beta\delta x_w+(\delta^{-1}+\beta\gamma)x_{wt}.
\end{align}
The equations (\ref{al:xst}) and (\ref{al:xsts}) give a second order linear
recurrence for the coefficients $x_1,x_s,x_{st},x_{sts},\ldots$ given
in matrix form by
\begin{align}\label{al:M}
\left[\begin{array}{c}
x_{(st)^{n+1}}\\
x_{(st)^{n+1}s}\end{array}\right] 
= & M
\left[\begin{array}{c}
x_{(st)^{n}}\\
x_{(st)^{n}s}\end{array}\right] \hspace{.3in}\mbox{where}\hspace{.3in}
M= \left[\begin{array}{cc}
\delta^{-1} & \beta\\
\gamma\delta^{-1} & \beta\gamma+\delta\end{array}\right] 
\end{align}
and the equations (\ref{al:xts}) and (\ref{al:xtst}) yield a recurrence
for the coefficients $x_1,x_t,x_{ts},x_{tst},\ldots$ given by 
\begin{align}\label{al:N}
\left[\begin{array}{c}
x_{(ts)^{n+1}}\\
x_{(ts)^{n+1}t}\end{array}\right] 
= & N \left[\begin{array}{c}
x_{(ts)^{n}}\\
x_{(ts)^{n}t}\end{array}\right] \hspace{.3in}\mbox{where}\hspace{.3in}
N= \left[\begin{array}{cc}
\delta & \gamma\\
\beta\delta & \beta\gamma+\delta^{-1}\end{array}\right] 
\end{align}
for $n=0,1,2,\ldots$.  We let $\bm$ and $\bn$ denote the initial
vectors  
\[ \bm=\left[\begin{array}{c}
x_1\\
x_s\end{array}\right] \;\;\mbox{and}\;\;
\bn=\left[\begin{array}{c}
x_1\\
x_t\end{array}\right]\]
of these recurrences.  They are constrained only by the single
equation (\ref{al:xinit})
\[\frac{x_s}{\alpha_s}-\frac{x_t}{\alpha_t}=\left(\mu-\frac{1}{1+q_s}+\frac{1}{1+q_t}\right)x_1.\]

Note that the matrices $M$ and $N$ from (\ref{al:M}) and (\ref{al:N})
have the same trace and determinant 
\[\tr M=\tr N=\beta\gamma+\delta+\delta^{-1}\;\;\mbox{and}\;\;
\det M=\det N=1,\]
hence they have the same eigenvalues.  Moreover, these eigenvalues are
multiplicative inverses of each other.  The basic fact we shall use to
eliminate most of the possible eigenvectors for $a_s-a_t$ is that
a nonzero solution $\bv=\sum_w x_w\ttau_w$ to the recurrence (\ref{al:M}) (and
similarly for (\ref{al:N})) must satisfy $M^n\bm\rightarrow 0$ as
$n\rightarrow\infty$.  Otherwise, the sum 
\[\sum_{n=0}^{\infty}\|M^n\bm\|^2=\sum_{n=0}^{\infty} (|x_{(st)^n}|^2+|x_{(st)^ns}|^2),\]
which is a lower bound for $\|\bv\|^2=\sum_{w}|x_w|^2$, will diverge.   

First we rule out the case where $M$ and $N$ do not have a basis of
eigenvectors.  In particular, $M$ and $N$ will only have one
eigenvalue in this case, and it will be equal to $+1$ or $-1$. 

\begin{lemma}\label{lem:equal-eigen}
If $M$ (and hence $N$) does not have linearly independent eigenvectors
and the initial vectors $\bm$ and $\bn$ are not both zero, then
$\sum_{w}|x_w|^2=\infty$.   
\end{lemma}

\begin{proof}{}
Without loss of generality, we can assume that $\bm$ is nonzero.  Let
$\chi\in\{1,-1\}$ be the eigenvalue for $M$.  Since the $\chi$- 
eigenspace for $M$ is $1$-dimensional, the Jordan form for
$M$ will be upper triangular with $\chi$ on the
diagonal and a $1$ in the upper corner.  It follows that there exists
a basis $\{\bm_1,\bm_2\}$ such that 
\[M^n\bm_1=\chi^n\bm_1,\;\;\mbox{and}\;\;
M^n\bm_2=\chi^n\bm_2+n\chi^{n-1}\bm_1.\] 
Writing $\bm=a\bm_1+b\bm_2$, we then have 
\[M^n\bm=(a\chi+bn)\chi^{n-1}\bm_1+b\chi^n\bm_2.\]
Since $a$ and $b$ are not both zero and $\chi=\pm 1$, the sequence
$M^n\bm$ does not converge to zero.
\end{proof}

Now assume $M$ and $N$ each have linearly independent eigenvectors
$\bm_1,\bm_2$ and $\bn_1,\bn_2$, respectively.  Since $M$ and $N$ have
the same eigenvalues, we can assume further that $\bm_i$ and $\bn_i$
correspond to the same eigenvalue, which we denote by $\chi_i$. Since
$\chi_1\chi_2=1$, we also assume $|\chi_1|\geq 1\geq|\chi_2|>0$.  
Our next step is to rule out the case where either of the initial
vectors has a nonzero component in the direction of the
$\chi_1$-eigenvector. 

\begin{lemma} \label{lem:bigger-eigen}
Assume the initial vectors $\bm$ and $\bn$ are
expressed as linear combinations of $\{\bm_1,\bm_2\}$ and
$\{\bn_1,\bn_2\}$, respectively.  If  $\bm$ has a nonzero component in
the direction of $\bm_1$ or  $\bn$ has a nonzero component in the
direction of $\bn_1$ then $\sum_{w}|x_w|^2=\infty$.   
\end{lemma}

\begin{proof}{}
Suppose  $\bm=a\bm_1+b\bm_2$ with $a\neq 0$.  Then 
\[M^n\bm=a(\chi_1)^n\bm_1+b(\chi_2)^n\bm_2.\]
Since $|\chi_1|\geq 1$, these vectors do not converge to zero.
The $\bn$ case is similar.
\end{proof}

In light of Lemmas~\ref{lem:equal-eigen} and \ref{lem:bigger-eigen},
we may assume that if $\bv=\sum_{w}x_w\ttau_w$ is an eigenvector of
$a_s-a_t$ with eigenvalue $\mu$, then 
\begin{itemize}
\item $M$ (and also $N$) has distinct eigenvalues $\chi_1$ and $\chi_2$ with
  $|\chi_1|>1>|\chi_2|$, and 
\item $\bm$ (respectively, $\bn$) is a $\chi_2$-eigenvector of $M$
  (resp., $N$).
\end{itemize}

We consider the following two cases.  
\begin{enumerate}
\item[]{\bf Case 1.} Either $\beta=0$ and $\chi_2=\delta^{-1}$ or $\gamma=0$
  and $\chi_2=\delta$.
\item[]{\bf Case 2.} The vectors 
\[
\bm'=\left[\begin{array}{c} \beta\\
    \chi_2-\delta^{-1}\end{array}\right]\;\;\mbox{and}\;\; 
\bn'= \left[\begin{array}{c} \gamma\\
    \chi_2-\delta\end{array}\right]
\]
are both nonzero.
\end{enumerate}

We first rule out Case 1.  Suppose $\beta=0$ and
$\chi_2=\delta^{-1}$.  Since $\beta=0$, the matrices $M$ and $N$
simplify to  
\[M=\left[\begin{array}{cc}\delta^{-1} & 0\\\gamma\delta^{-1}
    &\delta\end{array}\right]\;\;\mbox{and}\;\;
    N=\left[\begin{array}{cc}\delta & \gamma\\ 0
    &\delta^{-1}\end{array}\right],\]
and 
\[\mu=\frac{\alpha_{st}}{\alpha_s\alpha_t}=\frac{q_sq_t-1}{(q_s+1)(q_t+1)}.\]
Since $\chi_2=\delta^{-1}$, a calculation then shows that the
    $\chi_2$-eigenvectors of $M$ and $N$ are   
\[\left[\begin{array}{c} q_s-q_t\\
    -2\sqrt{q_s}(1+q_t)\end{array}\right]\;\;\mbox{and}\;\;
    \left[\begin{array}{c} -2\sqrt{q_t}(1+q_s)\\ 
    q_s-q_t\end{array}\right],\] 
respectively. Since $q_s$ and $q_t$ are positive reals, the first
    coordinates of these vectors cannot both be zero.  On the other
    hand, since these vectors are nonzero multiples of $\bm$ and $\bn$
    (which both have first coordinate equal to $x_1$), neither of
    these two vectors can have vanishing first coordinate.  It follows
    that $x_1\neq 0$, so we can scale $\nu$ so that $x_1=1$.  Then 
    $\bm=\left[\begin{array}{c}1\\x_s\end{array}\right]$ and 
    $\bn=\left[\begin{array}{c}1\\x_t\end{array}\right]$.  Since these
    are multiples of the $\chi_2$-eigenvectors above, we have 
\[x_s=-\frac{2\sqrt{q_s}(q_t+1)}{q_s-q_t},\]
and 
\[x_t=-\frac{q_s-q_t}{2\sqrt{q_t}(1+q_s)}.\]
Substituting these values into the initial equation (\ref{al:xinit}),
and isolating the numerator, we obtain 
\[(q_s+q_t+2)(2q_sq_t+q_s+q_t)=0\]
which has no solutions for positive $q_s$ and $q_t$.  A similar
analysis yields a contradiction in the case $\gamma=0$ and
$\chi_2=\delta$. 

For Case 2, the vectors $\bm'$ and $\bn'$ are nonzero.  A calculation
shows that they are $\chi_2$-eigenvectors for $M$ and $N$,
respectively, hence are nonzero multiples of $\bm$ and $\bn$.  We can
assume that $\beta$ and $\gamma$ are not both zero.  (Otherwise, both
$M$ and $N$ would be diagonal with entries $\delta$ and $\delta^{-1}$,
which means $\chi_2$ would have to be one of these, putting us back
into Case 1.)  Moreover, since $\bm'$ and $\bn'$ are nonzero
multiples of the vectors $\bm$ and $\bn$, respectively, and the latter
both have the same first coordinate $x_1$, we know that neither
$\beta$ nor $\gamma$ can be zero.  Again, by scaling $\nu$ if
necessary to get $x_1=1$, we then have    
\begin{align}\label{al:xs}
x_s=\frac{\chi_2-\delta^{-1}}{\beta}
\end{align}
and since  $\displaystyle \bn=\left[\begin{array}{c} 
x_1\\ x_t\end{array}\right]$ is a multiple of $\bn_2$, we have 
\begin{align}\label{al:xt}
x_t=\frac{\chi_2-\delta}{\gamma}.
\end{align}
Substituting these values into the initial equation (\ref{al:xinit})
we obtain 
\[\frac{\chi_2-\delta^{-1}}{\beta\alpha_s}-\frac{\chi_2-\delta}{\gamma\alpha_t}=\mu-\frac{1}{1+q_s}+\frac{1}{1+q_t}.\]
On the other hand, $\chi_2$ must also satisfy the characteristic
equation for $M$ and $N$, which is 
\[\chi_2^2-(\beta\gamma+\delta+\delta^{-1})\chi_2+1=0\]
Rewriting these equations in terms of $q_s$ and $q_t$, and solving
simultaneously for $\chi_2$ and $\mu$, we obtain the solutions
\begin{itemize}
\item $\chi_2=\sqrt{q_sq_t}$ and $\mu=0$,
\item $\chi_2=1/\sqrt{q_sq_t}$ and $\mu=0$,
\item $\chi_2=-\sqrt{q_s}/\sqrt{q_t}$ and $\mu=1$, or 
\item $\chi_2=-\sqrt{q_t}/\sqrt{q_s}$ and $\mu=-1$.
\end{itemize}
It follows that the only possible eigenvalues for $a_s-a_t$ are
$\mu=0$ and $\mu=\pm 1$, and hence (by Lemma~\ref{lem:lambda2mu}), the
only possible eigenvalues for $st$ are $\lambda=+1$ (if $\mu=0$) and
$\lambda=-1$ (if $\mu=\pm 1$).  

To describe the corresponding eigenvectors in a concise way, we define for any
real parameters $r_s,r_t$ the vector $\bk(r_s,r_t)$ as follows.  For
each $w\in W$, we define the coefficient $r^w$ as we did $q^w$.  For
the dihedral group, this looks like 
\begin{align}
\label{al:rw}
r^w=\left\{\begin{array}{ll}
r_s^nr_t^n & \mbox{if $w=(st)^n$ or $w=(ts)^n$}\\
r_s^{n+1}r_t^n & \mbox{if $w=(st)^ns$}\\
r_s^nr_t^{n+1} & \mbox{if $w=t(st)^n$}\end{array}\right.
\end{align}
for all $n\geq 0$.  We then define $\bk(r_s,r_t)$ by 
\[\bk(r_s,r_t)=\sum_{w}r^w\tau_w.\]
The $L^2$-norm of $\bk(r_s,r_t)$ is given by the geometric series 
\begin{align*}
\nonumber \|\bk(r_s,r_t)\|^2 &=\sum_w (r^w)^2q^w\\
\nonumber
&=1+r_s^2q_s+r_t^2q_t+\sum_{n=1}^{\infty}(2+r_s^2q_s+r_t^2q_t)(r_sr_t)^{2n}(q_sq_t)^n.
\end{align*}
This series converges if and only if 
\[(r_sr_t)^2<\frac{1}{q_sq_t},\]
and in this case converges to 
\begin{align} \label{al:knorm}
\|\bk(r_s,r_t)\|^2 &= \frac{(1+r_s^2q_s)(1+r_t^2q_t)}{1-r_s^2r_t^2q_sq_t}.
\end{align}
Putting all of this together, we obtain the following theorem. 

\begin{theorem}\label{thm:W-eigen}
If $\lambda$ is an eigenvalue for right or left multiplication
by $st$ on $\LW$, then $\lambda\in\{-1,+1\}$ and the corresponding
eigenspace is spanned by a single vector.  The eigenvalue/eigenvector
pairs occur as follows:
\begin{enumerate}
\item If $q_sq_t<1$, then $\lambda =1$ occurs with eigenvector $\bk(1,1)$,
\item If $q_sq_t>1$, then $\lambda =1$ occurs with eigenvector
  $\bk(-1/q_s,-1/q_t)$, 
\item If $q_s<q_t$, then $\lambda =-1$ occurs with eigenvector  
  $\bk(1,-1/q_t)$, 
\item If $q_s>q_t$, then $\lambda =-1$ occurs with eigenvector  
  $\bk(-1/q_s,1)$.
\end{enumerate}
\end{theorem} 

\begin{proof}{}
For right multiplication by $st$, the only thing left to prove is that
the indicated eigenvectors are the solutions to the recurrences
(\ref{al:M}) and (\ref{al:N}) for the 
given values of $\lambda$ and $\bq$.  Using the initial vectors
$\displaystyle \bm=\left[\begin{array}{c}  
1\\ x_s\end{array}\right]$, $\displaystyle \bn=\left[\begin{array}{c} 
1\\ x_t\end{array}\right]$ to get
\begin{align*}
x_{(st)^n} &=(\chi_2)^n,\\
x_{(ts)^n}& =(\chi_2)^n,\\
x_{(st)^ns}&=(\chi_2)^n x_s,\\
x_{(ts)^nt}&=(\chi_2)^n x_t,\\
\end{align*}
with $x_s$ and $x_t$ given by (\ref{al:xs}) and (\ref{al:xt}).  If,
for example, $\lambda=1$ and $q_sq_t<1$, then $\mu=0$ and
$\chi_2=\sqrt{q_sq_t}$.  It follows that $x_1=1$, $x_s=\sqrt{q_s}$,
$x_t=\sqrt{q_t}$, and in general $x_w=\sqrt{q^w}$.  Hence 
\[\bv=\sum_w \sqrt{q^w}\ttau_w=\sum_w \tau_w=\bk(1,1),\]
which is in $\LW$.  The cases (2)-(4) are similar.

For left multiplication, one notes that $\nu$ is a
$\lambda$-eigenvector for right multiplication by $(st)$ if and only
if $\nu^*$ is a $\bar{\lambda}$-eigenvector for left multiplication by
$(ts)=(st)^*$.  But since $|\lambda|=1$, this is true if and only if
$\nu^*$ is a $\lambda$-eigenvector for left multiplication by 
$st=(ts)^{-1}$.  The result then follows from the fact that for real
values of $r_s$ and $r_t$, $\kappa(r_s,r_t)$ is self-adjoint. 
\end{proof}

\section{Decompositions of $\NG$ and $\NW$-modules}

In this section we use the eigenspaces for the $st$-action to obtain
orthogonal decompositions of $\LG$ and $\LW$.  We then use these
decompositions to decompose any $\NW$-module in order to relate
its von Neumann dimension as an $\NG$-module to its dimension as an
$\NW$-module.  

First we describe key properties of the eigenvectors in
Theorem~\ref{thm:W-eigen}.  For a given $\bq$, we let $\bk_+$ denote
the vector  
\[\bk_+=\left\{\begin{array}{ll}
\bk(1,1) & \mbox{if $q_sq_t<1$}\\
\bk(-1/q_s,-1/q_t) & \mbox{if $q_sq_t>1$}\\
0 & \mbox{if $q_sq_t=1$}\end{array}\right.\]
and we let $\bk_-$ denote the vector 
\[\bk_-=\left\{\begin{array}{ll}
\bk(1,-1/q_t) & \mbox{if $q_s<q_t$}\\
\bk(-1/q_s,1) & \mbox{if $q_s>q_t$}\\
0 & \mbox{if $q_s=q_t$}\end{array}\right..\]

\begin{remark}
Many of the results of this section follow from results of Davis et al.
\cite{DDJO}.  In particular, for $q_sq_t<1$ the span of $\bk_+$ is
the invariant subspace of $\LW$ consisting of constants, which is denoted
by $A^{\{s,t\}}$ in \cite{DDJO}.  Projection onto this subspace is the
averaging operator denoted by $a_{\{s,t\}}$ in \cite{DDJO} and by
$\tk_+$, below.   The vectors $\bk_{\pm}$ for other values of $\bq$
can all be obtained from $\bk_+$ by applying the ``partial
$j$-automorphisms'' of $\LW$ described in \cite[Section~9]{OS}.  For
completeness, we present proofs here without using these more
general results.
\end{remark}

\begin{proposition}\label{prop:Wfixesks}
Any element $w\in W$ fixes the vectors $\bk_+$ and $\bk_-$ (up
to sign).  More precisely, we have:
\begin{enumerate}
\item $s\bk_+=\bk_+s=\bk_+$ and
  $t\bk_+=\bk_+t=\bk_+$ if $q_sq_t<1$, 
\item $s\bk_+=\bk_+s=-\bk_+$ and
  $t\bk_+=\bk_+t=-\bk_+$ if $q_sq_t>1$, 
\item $s\bk_-=\bk_-s=\bk_-$ and
  $t\bk_-=\bk_-t=-\bk_-$ if $q_s<q_t$, and 
\item $s\bk_-=\bk_-s=-\bk_-$ and
  $t\bk_-=\bk_-t=\bk_-$ if $q_s>q_t$.
\end{enumerate}
\end{proposition} 

\begin{proof}{}
These are all calculations using Hecke multiplication.  The two basic
identities one needs are $sa_s=a_s$ and $sh_s=-h_s$.  These follows
from the definitions of $a_s$ and $h_s$ in (\ref{al:as}) and
(\ref{al:hs}) in terms of the group algebra basis:
\[sa_s=\frac{s(1+s)}{2}=\frac{s+s^2}{2}=\frac{s+1}{2}=a_s,\]
and 
\[sh_s=\frac{s(1-s)}{2}=\frac{s-s^2}{2}=\frac{s-1}{2}=-h_s.\]
Rewriting these identities using the expressions for $a_s$ and $h_s$
using the Hecke algebra basis in (\ref{al:as}) and
(\ref{al:hs}), and multiplying both sides by $1+q_s$, we
obtain the identities 
\begin{align}
s(1+\tau_s)=1+\tau_s \;\;\mbox{and}\;\;
s(q_s-\tau_s)=-(q_s-\tau_s).
\end{align}
Now to get, for example, the identity $s\bk_+=\bk_+$ when $q_sq_t<1$,
we have 
\begin{align*}
s\bk_+ & = s\bk(1,1)\\
 &= s(1+\tau_s+\tau_t+\tau_{st}+\tau_{ts}+\tau_{sts}+\cdots)\\
 &= s(1+\tau_s)(1+\tau_t+\tau_{ts}+\cdots)\\
 &= (1+\tau_s)(1+\tau_t+\tau_{ts}+\cdots)\\
 &= \bk_+.
\end{align*}
To get the identity $s\bk_-=-\bk_-$ when $q_s>q_t$, we have 
\begin{align*}
s\bk_- & = s\bk(-1/q_s,1)\\
 &= s(1-\tau_s/q_s+\tau_t-\tau_{st}/q_s-\tau_{ts}/q_s+\tau_{sts}/q_s^2-\cdots)\\
 &= s(q_s-\tau_s)(1/q_s+\tau_t/q_s-\tau_{ts}/q_s^2-\cdots)\\
 &= -(q_s-\tau_s)(1/q_s+\tau_t/q_s+\tau_{ts}/q_s^2-\cdots)\\
 &= - \bk_-.
\end{align*}
The remaining identities are obtained in a similar fashion by
factoring $(1+\tau_s)$, $(1+\tau_t)$, $(q_s-\tau_s)$, or
$(q_t-\tau_t)$ out of $\bk_{\pm}$ on the right or left depending on
the case.  We leave the details to the reader.
\end{proof}

Solving for $\tau_s$ in (\ref{al:as}) we get the formulas 
\[\tau_s=\frac{q_s-1}{2}+\frac{q_s+1}{2}s\;\;\mbox{and}\;\;
\tau_t=\frac{q_t-1}{2}+\frac{q_t+1}{2}t.\]
Using Proposition~\ref{prop:Wfixesks}, we then obtain additional
formulas for products $\bk_{\pm}$ with the Hecke generators  
$\tau_s$ and $\tau_t$:
\begin{align}
\nonumber\tau_s\bk_+=\bk_+\tau_s=q_s\bk_+\;\;\mbox{and}\;\;
  \tau_t\bk_+=\bk_+\tau_t=q_t\bk_+ &\;\;  \mbox{if}\;\;q_sq_t<1,\\
\label{al:k-times-tau}\tau_s\bk_+=\bk_+\tau_s=-\bk_+\;\;\mbox{and}\;\;
  \tau_t\bk_+=\bk_+\tau_t=-\bk_+ &\;\;  \mbox{if}\;\;q_sq_t>1,\\
\nonumber\tau_s\bk_-=\bk_-\tau_s=q_s\bk_-\;\;\mbox{and}\;\;
  \tau_t\bk_-=\bk_-\tau_t=-\bk_- &\;\;  \mbox{if}\;\;q_s<q_t,\\
\nonumber\tau_s\bk_-=\bk_-\tau_s=-\bk_-\;\;\mbox{and}\;\;
  \tau_t\bk_-=\bk_-\tau_t=q_t\bk_- &\;\;  \mbox{if}\;\;q_s>q_t.
\end{align}
These are useful because they allow us to show that the vectors
$\bk_{\pm}$ extend to well-defined operators in $\NW$.

\begin{proposition}\label{prop:ksbounded}
The elements $\bk_+$ and $\bk_-$ acting on $\bbR W$ extend to bounded
operators in $\NW$ (and $\NG$). 
\end{proposition}

\begin{proof}{}
Let $\bk$ be either $\bk_+$ or $\bk_-$.   Since $\bk$ commutes with
all elements in $\bbR W$, it suffices to show that for any $y\in\bbR
W$, we have $\|\bk y\|_{\bq}\leq C \|y\|_{\bq}$ for some constant
$C$.  In fact, we'll show that $C=\|\bk\|_{\bq}^2$ works.  By
definition, $\bk$ is one of the four vectors $\bk(r_s,r_t)$ where
$(r_s,r_t)$ is one of the pairs $(1,1),(-1/q_s,-1/q_t),
(1,-1/q_t),(-1/q_s,1)$; hence,  
\[\bk=\sum r^w\tau_w\]
with $r^w$ given by (\ref{al:rw}).  Expressing $\tau_w$ as a product
of $\tau_s$'s and $\tau_t$'s, and using the product formulas
(\ref{al:k-times-tau}), one can verify that   
\begin{align}\label{al:tau-w-k}
\bk\tau_w=q^w r^w\bk.
\end{align}
Letting $y=\sum_w y_w\tau_w$, we then have 
\[\bk y=\sum_w y_w\bk\tau_w=\sum_w y_w q^w r^w
\bk=\sum_w(y_w\sqrt{q^w})(r^w\sqrt{q^w})\bk.\]
Taking square norms, we have
\begin{align*}
\|\bk y\|_{\bq}^2 & =|\sum_w(y_w\sqrt{q^w})(r^w\sqrt{q^w})|^2\|\bk\|_{\bq}^2\\
&\leq \sum_w|y_w\sqrt{q^w}|^2\sum_w|r^w\sqrt{q^w}|^2\|\bk\|_{\bq}^2\\
&= (\sum_w|y_w|^2q^w)(\sum_w|r^w|^2q^w)\|\bk\|_{\bq}^2\\
&= \|y\|_{\bq}^2\|\bk\|_{\bq}^2\|\bk\|_{\bq}^2,
\end{align*}
and taking square roots gives $\|\bk y\|_{\bq}\leq \|\bk\|_{\bq}^2\|y\|_{\bq}$.
\end{proof}

Let $K_+$ and $K_-$ denote the $+1$ and $-1$-eigenspaces
(respectively) for the right $st$-action on $\LW$.  In light of
Theorem~\ref{thm:W-eigen}, $K_+$ (respectively, $K_-$) is spanned by
the single vector $\bk_+$ (resp., $\bk_-$).    

\begin{proposition}\label{prop:G-decomp}
The subspace $\LG\subseteq \LW$ is $st$-invariant (on both sides) 
and contains both $K_+$ and $K_-$.  In fact, we have an orthogonal
decomposition of $\NG$-modules given by 
\[\LG=K_+\oplus K_-\oplus K_{\emptyset}\]
where $K_{\emptyset}$ is the orthogonal complement of $K_+\oplus K_-$ in
$\LG$.
\end{proposition}

\begin{proof}{}
That $\LG$ is $st$-invariant is clear, as is the orthogonality of
$K_+$ and $K_-$ ($st$ is a unitary operator so its eigenspaces
are orthogonal).  It only remains to prove then that $K_{\pm}\subseteq
\LG$.  For this, we use the fact that the orthogonal projection $\pi$
from $\LW$ onto $\LG$ is an $\NG$-module map, hence commutes with
multiplication by $st$.  It follows that $\pi$ must map $st$-eigenspaces to
$st$-eigenspaces (with the same eigenvalue).   Since $K_+$ is spanned by the single
vector $\bk_+$ we must have either $\pi(\bk_+)=\bk_+$ or $\pi(\bk_+)=0$.  In
other words, $\bk_+$ is either in the subspace $\LG$ or it is
orthogonal to it.  To be orthogonal to $\LG$, one would have to
have $\<\bk_+,1\>_{\bq}=0$ since $1\in\bbR G\subseteq\LG$.  But it
follows immediately from the definition of $\bk_+$ that either
$\bk_+$ is zero (in which case $K_+\subseteq\LG$, trivially) or
$\<\bk_+,1\>_{\bq}=1$.  Hence $\bk_+\in\LG$ and so $K_+\subseteq\LG$.
The same argument applied to the $-1$-eigenspace for $st$ shows that
$K_-\subseteq\LG$. 
\end{proof}

It will be convenient to work with the orthogonal projections onto $K_+$
and $K_-$.  Since $K_+$ and $K_-$ are the spans of the single vectors
$\bk_+$ and $\bk_-$, the relevant projections are simply given by 
appropriate scalings.  We define $\tk_+$ and $\tk_-$ by 
\[\tk_+=\frac{\bk_+}{\|\bk_+\|_{\bq}^2}\;\;\mbox{and}\;\;
\tk_-=\frac{\bk_-}{\|\bk_-\|_{\bq}^2},\]
and we define $\tu$ by 
\[\tu=1-\tk_+-\tk_-.\]

\begin{proposition}\label{prop:idempotents}
The elements $\tk_{\pm}$ and $\tu$ are central self-adjoint
idempotents in the von Neumann algebras $\NG$ and $\NW$.  In particular,
multiplication on the right or left by $\tk_{\pm}$ defines orthogonal
projection from $\LG$ onto $K_{\pm}$ and multiplication by $\tu$
defines orthogonal projection from $\LG$ onto $K_{\emptyset}$.
\end{proposition}

\begin{proof}{}
Since $\tk_{\pm}$ are multiples of $\bk_{\pm}$, by
Proposition~\ref{prop:ksbounded} they are elements of $\NG$ and $\NW$.  Since $\tu$ is a finite linear
combination of $1$, $\tk_+$ and $\tk_-$, it is in $\NG$ and $\NW$ as
well. Since all three of these operators commute with every element of
$\bbR W$ (by Proposition~\ref{prop:Wfixesks}) and $\bbR W$ is dense in
$\LW$, they are all central.  Self-adjointness follows from the
explicit formulas for $\bk_+$ and $\bk_-$, in which the coefficient
of $\tau_w$ is always the same as the coefficient of
$\tau_w^*=\tau_{w^{-1}}$.  It remains to show that they are all
idempotent. If $\bk$ denotes $\bk_+$ or $\bk_-$, then we have 
\[\bk=\sum_w r^w\tau_w\]
with $r^w$ given by (\ref{al:rw}), hence by (\ref{al:tau-w-k}) we have 
\[\bk^2=\sum_w r^w\tau_w\bk=\sum_w (r^w)^2q^w\bk=\|k\|_{\bq}^2\bk.\]
Dividing both sides by $\|\bk\|_{\bq}^2$ gives $\tk^2=\tk$.
The operator $\tu=1-\tk_+-\tk_1$ is idempotent because it is the 
orthogonal projection onto the complement of $K_+$ and $K_-$.
\end{proof}

Using these idempotents, we can compute $\NG$-dimensions of the various
pieces in our decomposition.

\begin{lemma}\label{lem:Gdims}
The von Neumann dimensions of the $\NG$-modules $K_+$, $K_-$, and
$K_{\emptyset}$ are all piecewise rational functions of the form 
\[n_{\emptyset}+\frac{n_s}{1+q_s}+\frac{n_t}{1+q_t}.\]
where $n_{\emptyset},n_s,n_t$ are piecewise constant integer functions
of $\bq$.  More precisely, we have
\[\dim^{\bq}_G K_+=\frac{|1-q_sq_t|}{(1+q_s)(1+q_t)},\hspace{.5in}
\dim^{\bq}_G K_-=\frac{|q_t-q_s|}{(1+q_s)(1+q_t)},\]
and
\[\dim^{\bq}_G K_{\emptyset}=\left\{\begin{array}{ll}
\displaystyle\frac{2q_s}{1+q_s} & \mbox{if $q_sq_t\leq 1$ and $q_s\leq q_t$,}\\
\displaystyle\frac{2q_t}{1+q_t} & \mbox{if $q_sq_t\leq 1$ and $q_s\geq q_t$,}\\
\displaystyle\frac{2}{1+q_t} & \mbox{if $q_sq_t\geq 1$ and $q_s\leq q_t$,}\\
\displaystyle\frac{2}{1+q_s} & \mbox{if $q_sq_t\geq 1$ and $q_s\geq q_t$.}\end{array}\right.\]
\end{lemma}

\begin{proof}{}
By definition of von Neumann dimension and the idempotents
$\tk_{\pm}$, we have  
\[\dim_G^{\bq} K_{\pm}=\<\tk_{\pm},1\>_{\bq}=\frac{1}{\|\bk_{\pm}\|_{\bq}^2}\<\bk_{\pm},1\>_{\bq}=\frac{1}{\|\bk_{\pm}\|_{\bq}^2}.\]
Substituting $(r_s,r_t)=(1,1)$ and $(r_s,r_t)=(-1/q_s,-1/q_t)$ into
(\ref{al:knorm}) to get $\|\bk_+\|_{\bq}^2$, we obtain 
\begin{align}\label{al:dimK+}
\dim^{\bq}_G K_+ =\<\tk_+,1\>_{\bq}=\frac{|1-q_sq_t|}{(1+q_s)(1+q_t)},
\end{align}
and substituting $(r_s,r_t)=(1,-1/q_t)$ and $(r_s,r_t)=(-1/q_t,1)$ into
(\ref{al:knorm}) to get $\|\bk_-\|_{\bq}^2$, we obtain 
\begin{align}\label{al:dimK-}
\dim^{\bq}_G K_-=\<\tk_-,1\>_{\bq}=\frac{|q_t-q_s|}{(1+q_s)(1+q_t)}.
\end{align}
Since $K_{\emptyset}$ is the orthogonal complement of $K_+$ and $K_-$
in $\LG$ and $\dim_G^{\bq}\LG=1$, we have 
\[\dim^{\bq}_GK_{\emptyset}=1-\frac{|1-q_sq_t|}{(1+q_s)(1+q_t)}-\frac{|q_t-q_s|}{(1+q_s)(1+q_t)}\]
which simplifies to the given formulas in the four cases indicated.

To see that all of these expressions are piecewise rational functions of the
indicated form, simply note that 
\[\frac{1-q_sq_t}{(1+q_s)(1+q_t)}=-1+\frac{1}{1+q_s}+\frac{1}{1+q_t},\]
\[\frac{q_t-q_s}{(1+q_s)(1+q_t)}= \frac{1}{1+q_s}-\frac{1}{1+q_t},\]
and 
\[\frac{2q}{1+q}=2-\frac{2}{1+q}.\]
\end{proof}

We now extend the orthogonal decomposition of $\LG$ to any Hilbert
$\NG$-module.  By Proposition~\ref{prop:G-decomp}, we can identify
$\LG^n$ with the orthogonal sum $(K_+)^n\oplus (K_-)^n\oplus (K_{\emptyset})^n$.

\begin{proposition}\label{prop:V-G-decomp} 
Let $V\subseteq \LG^n$ be a closed subspace that is invariant with
respect to the diagonal left $\bbR G$-action, and let $V_+=\tk_+ V$,  
$V_-=\tk_-V$, $V_{\emptyset}=\tu V$.  Then we have an orthogonal
decomposition 
\[V=V_+\oplus V_-\oplus V_{\emptyset}\]
with $V_+\subseteq (K_+)^n$, $V_-\subseteq (K_-)^n$ and
$V_{\emptyset}\subseteq (K_{\emptyset})^n$.
\end{proposition}  

\begin{proof}{}
By Proposition~\ref{prop:idempotents}, $\tk_+$, $\tk_-$, and $\tu$
are all elements of $\NG$ and define orthogonal projections from $\LG$
onto $K_+$, $K_-$, and $K_{\emptyset}$, respectively.  It follows that
diagonal left multiplication by these elements on $\LG^n$ defines
orthogonal projection onto the subspaces $(K_+)^n$, $(K_-)^n$,
$(K_{\emptyset})^n$, respectively.  It follows that the summands $V_+$,
$V_-$, $V_{\emptyset}$ are orthogonal.  Since $V$ is a left $\NG$-module,
each of the summands $V_+$, $V_-$, and $V_{\emptyset}$ must be
contained in $V$, so we have 
\[V\supseteq V_+\oplus V_-\oplus V_{\emptyset}.\]
On the other hand, since $1=\tk_++\tk_-+\tu$, we know that
$x=\tk_+x+\tk_-x+\tu x$ for any $x\in V$, giving us the opposite
inclusion.
\end{proof}

To extend our decomposition of $\LG$ to a decomposition of $\LW$, we
note that $\LW$ is spanned by $\LG$ and its translate $\LG s$.  By
Proposition~\ref{prop:Wfixesks}, both $K_+$ and $K_-$ are also
contained in $\LG s$, suggesting the following decomposition for
$\LW$. 

\begin{proposition}\label{prop:W-decomp}
We have an orthogonal decomposition of $\NG$-modules given by  
\[\LW=K_+\oplus K_-\oplus K_{\emptyset}\oplus K_{\emptyset}s.\]
Moreover, $K_{\emptyset}$ and $K_{\emptyset}s$ are isomorphic as $\NG$-modules.
\end{proposition}

\begin{proof}{}
Right multiplication by $s$ is a self-adjoint involution, hence an
isometry.  It follows that (1) $K_{\emptyset}$ maps isomorphically
(isometrically and equivariantly with respect to the left $\bbR W$-action) to
$K_{\emptyset}s$, and (2) preserves orthogonality in $\LW$.  The
latter implies that 
\[\LG s=(K_+\oplus K_-\oplus K_{\emptyset})s=(K_+s\oplus K_-s\oplus
    K_{\emptyset}s)=(K_+\oplus K_-\oplus
    K_{\emptyset}s),\]
where the last equality follows from Proposition~\ref{prop:Wfixesks}.
Since $\LW$ is spanned by $\LG$ and $\LG s$, we have
\begin{align*}
\LW &= \LG + \LG s\\
    &= (K_+\oplus K_-\oplus K_{\emptyset})+ (K_+\oplus K_-\oplus
    K_{\emptyset}s)\\
&=K_+\oplus K_-\oplus (K_{\emptyset}+ K_{\emptyset}s).
\end{align*}

The only thing left to prove is that $K_{\emptyset}$ and
$K_{\emptyset}s$ are orthogonal.  Since $G$ spans a dense subspace of
$\LG$, we know that $\{(st)^n\tk_{\emptyset}\;|\; n\in\bbZ\}$ spans a
dense subspace of $K_{\emptyset}$, and $\{(st)^ns\tk_{\emptyset}\;|\;
n\in\bbZ\}$ spans a dense subspace of $K_{\emptyset}s$.  It therefore
suffices to prove that 
\[\<(st)^n\tk_{\emptyset},(st)^ms\tk_{\emptyset}\>_{\bq}=0\]
for all $m,n\in\bbZ$.  Using the fact that $\tk_{\emptyset}$ is a self
adjoint idempotent and $(st)^*=(st)^{-1}$, we have 
\begin{align}\label{al:IP}
\<(st)^n\tk_{\emptyset},(st)^ms\tk_{\emptyset}\>_{\bq}=
\<s(st)^{n-m}\tk_{\emptyset}^2,1\>_{\bq}=
\<s(st)^{n-m}\tk_{\emptyset},1\>_{\bq}.
\end{align}
But since $\tk_{\emptyset}$ is central, we have (for any $x\in\LW$)
\[\<sxs\tk_{\emptyset},1\>_{\bq}=\<xs\tk_{\emptyset},s\>_{\bq}=\<xs^2\tk_{\emptyset},1\>_{\bq}=\<x\tk_{\emptyset},1\>_{\bq}\]
and, similarly, 
\[\<txt\tk_{\emptyset},1\>_{\bq}=\<x\tk_{\emptyset},1\>_{\bq}.\]
Repeated applications of this identity then reduce (\ref{al:IP}) to 
\[\<(st)^n\tk_{\emptyset},(st)^ms\tk_{\emptyset}\>_{\bq}=\left\{
\begin{array}{ll}
\<s\tk_{\emptyset},1\>_{\bq}&\mbox{if $n-m$ is even,}\\
\<t\tk_{\emptyset},1\>_{\bq}&\mbox{if $n-m$ is odd.}\end{array}\right.\]
By definition of $\tk_{\emptyset}$ and
Proposition~\ref{prop:Wfixesks}, we have  
\begin{eqnarray*}
\<s\tk_{\emptyset},1\>_{\bq}&=&\<s,1\>_{\bq}-\<s\tk_+,1\>_{\bq}-
\<s\tk_-,1\>_{\bq}\\
&=&\<s,1\>_{\bq}-\sigma_1\<\tk_+,1\>_{\bq}-\sigma_2\<\tk_-,1\>_{\bq}
\end{eqnarray*}
where $\sigma_1$ is $+1$ (resp., $-1$) if $q_sq_t<1$
(resp. $q_sq_t>1$) and $\sigma_2$ is $+1$ (resp., $-1$) if $q_s<q_t$
(resp. $q_s>q_t$).  Since
$s=\frac{1-q_s}{1+q_s}+\frac{2}{1+q_s}\tau_s$, we have
$\<s,1\>_{\bq}=\frac{1-q_s}{1+q_s}$, and hence by (\ref{al:dimK+}) and
(\ref{al:dimK-}) we have 
\[\<s\tk_{\emptyset},1\>_{\bq}=\frac{1-q_s}{1+q_s}-\frac{1-q_sq_t}{(1+q_s)(1+q_t)}-\frac{q_t-q_s}{(1+q_s)(1+q_t)}=0.\]
A similar calculation gives
\[\<t\tk_{\emptyset},1\>_{\bq}=\frac{1-q_t}{1+q_t}-\frac{1-q_sq_t}{(1+q_s)(1+q_t)}-\frac{q_s-q_t}{(1+q_s)(1+q_t)}=0.\]
This completes the proof.
\end{proof}

We now extend our orthogonal decomposition of $\LW$ to any Hilbert
$\NW$-module. By Proposition~\ref{prop:W-decomp}, we can identify
$(\LW)^n$ with $(K_+)^n\oplus (K_-)^n\oplus(K_{\emptyset}\oplus K_{\emptyset} s)^n$.

\begin{proposition}\label{prop:V-W-decomp} 
Let $V\subseteq(\LW)^n$ be a closed
subspace that is invariant with respect to the diagonal left $\bbR
W$-action, and let $V_+=\tk_+ V$, $V_-=\tk_-V$, $V_{\emptyset}=\tu V$.
Then we have an orthogonal decomposition  
\[V=V_+\oplus V_-\oplus V_{\emptyset}\]
with $V_+\subseteq (K_+)^n$, $V_-\subseteq (K_-)^n$ and
$V_{\emptyset}\subseteq (K_{\emptyset}\oplus K_{\emptyset} s)^n$.
\end{proposition}  

\begin{proof}{}
The proof is the same as the proof of
Proposition~\ref{prop:V-G-decomp}.  The only difference is that as an
operator on $\LW$, the idempotent $\tk_{\emptyset}$ projects onto the
orthogonal complement of $K_+\oplus K_-$ {\em in $\LW$}, which is now
$K_{\emptyset}\oplus K_{\emptyset}s$.
\end{proof}

Any $\NW$-module is naturally an $\NG$-module, hence we can
ask for its von Neumann dimension with respect to either structure.
The following lemma relates the two.

\begin{lemma}\label{lem:dimG2W}
Let $V\subseteq(\LW)^n$ be a Hilbert $\NW$-module.  Then 
\begin{enumerate}
\item $\dim_W^{\bq}V_+=\dim_G^{\bq}V_+$,
\item $\dim_W^{\bq}V_-=\dim_G^{\bq}V_-$, and 
\item
  $\dim_W^{\bq}V_{\emptyset}=\frac{1}{2}\dim_G^{\bq}V_{\emptyset}$.
\end{enumerate}
\end{lemma}

\begin{proof}{}
We identify $(\LW)^n$ with $(K_+)^n\oplus (K_-)^n\oplus(K_{\emptyset}\oplus
K_{\emptyset}s)^n$.  To prove (1) and (2), let $\pi_+:(K_+)^n\rightarrow
(K_+)^n$ and $\pi_-:(K_-)^n\rightarrow (K_-)^n$ denote orthogonal
projections onto $V_+$ and $V_-$, respectively.  By composing
projections, we then have that the orthogonal projection from
$(\LW)^n$ to $V_+$, and hence from $\LG^n$ to $V+$, are both given by
$\pi_+\tk_+$.  Similarly, the orthogonal projection from 
$(\LW)^n$ to $V_-$ is given by $\pi_-\tk_-$.  Let
$\epsilon_1,\ldots,\epsilon_n$ be the standard  basis for $(\LW)^n$ as a
free $\NW$-module.  Then it can also be regarded as the standard basis
for the subspace $\LG^n$ regarded as a free $\NG$-module.  Hence, we
have 
\[\dim^{\bq}_GV_+=\sum_{i=1}^n\<\pi_+(\tk_+\epsilon_i),\epsilon_i\>=\dim^{\bq}_W
V_+,\]
and
\[\dim^{\bq}_GV_-=\sum_{i=1}^n\<\pi_-(\tk_-\epsilon_i),\epsilon_i\>=\dim^{\bq}_W
V_-.\]

To prove (3), we let $\pi_{\emptyset}:(K_{\emptyset}\oplus
K_{\emptyset}s)^n\rightarrow (K_{\emptyset}\oplus
K_{\emptyset}s)^n $ be orthogonal projection onto $V_{\emptyset}$.
Again by composing projections, we have that the orthogonal
projection from $(\LW)^n$ to $V_{\emptyset}$ is given by
$\pi_{\emptyset}k_{\emptyset}$, and hence
\[\dim_W^{\bq}V_{\emptyset}=\sum_{i=1}^n\<\pi_{\emptyset}(k_{\emptyset}\epsilon_i),\epsilon_i\>.\]
To calculate the dimension of $V_{\emptyset}$ as an $\NG$-module, we
shall embed it in the free $\NG$-module $\LG^n\oplus\LG^n$.  We let
$\epsilon_1,\ldots,\epsilon_n$ denote the standard basis for the first
summand of $\LG^n\oplus\LG^n$ and $\epsilon_1',\ldots,\epsilon_n'$
denote the standard basis for the second summand.  
We then define 
\[\phi:(K_{\emptyset}\oplus
K_{\emptyset}s)^n\rightarrow\LG^n\oplus\LG^n\] 
by $\phi(x_1+x_1's,\ldots,x_n+x_n's)\mapsto
((x_1,\ldots,x_n),(x_1',\ldots,x_n'))$.  This map is an isometric
embedding, equivariant with respect to the left $\bbR G$-action, and
the image is $(K_{\emptyset})^n\oplus (K_{\emptyset})^n$.  As an $\NG$-module
$(K_{\emptyset}\oplus K_{\emptyset}s)^n$ is generated by
$\tk_{\emptyset}\epsilon_1,\ldots,\tk_{\emptyset}\epsilon_n$ and  
$\tk_{\emptyset}s\epsilon_1,\ldots,\tk_{\emptyset}s\epsilon_n$.
The images of these generators are given by
$\phi(\tk_{\emptyset}\epsilon_i)=\tk_{\emptyset}\epsilon_i$ and 
$\phi(\tk_{\emptyset}s\epsilon_i)=\tk_{\emptyset}\epsilon_i'$.
As an $\NG$-module $V_{\emptyset}$ is isomorphic to the image
$\phi(V_{\emptyset})\subseteq\LG^n\oplus\LG^n$, and orthogonal
projection onto this image is given by the composition
$\phi\pi_{\emptyset}\phi^{-1}\tk_{\emptyset}$.  We can therefore compute 
\begin{flalign*}
\dim_G^{\bq}V_{\emptyset} & = \dim_G^{\bq}\phi(V_{\emptyset})\\
&= \sum_{i=1}^{n}
\<\phi\pi_{\emptyset}\phi^{-1}\tk_{\emptyset}(\epsilon_i),\epsilon_i\>+
\sum_{i=1}^{n}
\<\phi\pi_{\emptyset}\phi^{-1}\tk_{\emptyset}(\epsilon_i'),\epsilon_i'\> 
 &&\\
&\hspace{.5in}\mbox{(definition of $\dim^{\bq}_G$)}\\ 
& = \sum_{i=1}^{n}
\<\phi\pi_{\emptyset}\phi^{-1}\tk_{\emptyset}^2(\epsilon_i),\epsilon_i\>+
\sum_{i=1}^{n}
\<\phi\pi_{\emptyset}\phi^{-1}\tk_{\emptyset}^2(\epsilon_i'),\epsilon_i'\>&&\\
&\hspace{.5in}\mbox{($\tk_{\emptyset}$ is idempotent)}\\
& = \sum_{i=1}^{n}
\<\tk_{\emptyset}\phi\pi_{\emptyset}\phi^{-1}\tk_{\emptyset}(\epsilon_i),\epsilon_i\>+
\sum_{i=1}^{n}
\<\tk_{\emptyset}\phi\pi_{\emptyset}\phi^{-1}\tk_{\emptyset}(\epsilon_i'),\epsilon_i'\>\\
&\hspace{.5in}\mbox{($\phi\pi_{\emptyset}\phi^{-1}\tk_{\emptyset}$ is $\NG$-equivariant)}\\
& = \sum_{i=1}^{n}
\<\phi\pi_{\emptyset}\phi^{-1}\tk_{\emptyset}(\epsilon_i),\tk_{\emptyset}\epsilon_i\>+
\sum_{i=1}^{n}
\<\phi\pi_{\emptyset}\phi^{-1}\tk_{\emptyset}(\epsilon_i'),\tk_{\emptyset}\epsilon_i'\>&&\\
&\hspace{.5in}\mbox{($\tk_{\emptyset}$ is self-adjoint)}\\
& = \sum_{i=1}^{n}
\<\phi\pi_{\emptyset}(\tk_{\emptyset}\epsilon_i),\phi(\tk_{\emptyset}\epsilon_i)\>+
\sum_{i=1}^{n}
\<\phi\pi_{\emptyset}(\tk_{\emptyset}s\epsilon_i),\phi(\tk_{\emptyset}s\epsilon_i\>&&\\
&\hspace{.5in}\mbox{(definition
  of $\phi$)}\\
& = \sum_{i=1}^{n}
\<\pi_{\emptyset}(\tk_{\emptyset}\epsilon_i),\tk_{\emptyset}\epsilon_i\>+
\sum_{i=1}^{n}
\<\pi_{\emptyset}(\tk_{\emptyset}s\epsilon_i),\tk_{\emptyset}s\epsilon_i\>&&\\
&\hspace{.5in}\mbox{($\phi$ is an isometry)}\\
& = \sum_{i=1}^{n}
\<\tk_{\emptyset}\pi_{\emptyset}(\tk_{\emptyset}\epsilon_i),\epsilon_i\>+
\sum_{i=1}^{n}
\<s\tk_{\emptyset}\pi_{\emptyset}(\tk_{\emptyset}s\epsilon_i),\epsilon_i\>\\
&\hspace{.5in}\mbox{($s$ and $\tk_{\emptyset}$ are self-adjoint)}\\
& = \sum_{i=1}^{n}
\<\pi_{\emptyset}(\tk_{\emptyset}^2\epsilon_i),\epsilon_i\>+
\sum_{i=1}^{n}
\<\pi_{\emptyset}(s\tk_{\emptyset}^2s\epsilon_i),\epsilon_i\>
&&\\
&\hspace{.5in}\mbox{($\pi_{\emptyset}$ is $\NW$-equivariant)}\\
& = \sum_{i=1}^{n}
\<\pi_{\emptyset}(\tk_{\emptyset}\epsilon_i),\epsilon_i\>+
\sum_{i=1}^{n}
\<\pi_{\emptyset}(\tk_{\emptyset}\epsilon_i),\epsilon_i\>\\
&\hspace{.5in}\mbox{($\tk_{\emptyset}$ is a central idempotent and $s^2=1$)}\\
&= 2\dim^{\bq}_WV_{\emptyset}.
\end{flalign*}

\end{proof}

\section{Kernels of $\bbR G$ and $\bbR W$-matrices}

In this section, we consider only those $\NG$-modules (respectively,
$\NW$-modules) that are given by kernels of right multiplication by
$\bbR G$-matrices (resp., $\bbR W$-matrices).  The fundamental fact
that our arguments rely on is that the submodules
$K_+,K_-,K_{\emptyset}\subseteq\LG$ are irreducible in the sense that right
multiplication by an element of $\bbR G$ is either the zero map or an
isomorphism.  For $K_+$ and $K_-$ this is obvious since they are each
spanned by a single vector, but for $K_{\emptyset}$ we need the fact that
there are no other $st$-eigenvectors in $\LG$.

\begin{proposition}\label{prop:Virred}
For any element $y\in\bbR G$, let
$R_y:K_{\emptyset}\rightarrow K_{\emptyset}$ denote
(right) multiplication by $y$.  Then 
\[\ker R_y=\left\{\begin{array}{ll}
K_{\emptyset}& \mbox{if $y=0$,}\\
0 & \mbox{if $y\neq 0$.}\end{array}\right.\]
\end{proposition}

\begin{proof}{}
Since $G$ is infinite cyclic generated by $st$, $y$ is a Laurent
polynomial in $st$, hence can be factored as 
\[y=C\cdot(st)^{-n}\cdot p(st)\]
where $n$ is an integer, $C$ is a nonzero real constant, and $p(z)$ is a
polynomial in $z$ with real coefficients.  Factoring this polynomial
gives 
\[y=C\; (st)^{-n}\;(st-\lambda_1)\cdots(st-\lambda_k),\]
where the $\lambda_i\in\bbC$ are the roots of $p(z)$. 
If $R_y(x)=0$ for some nonzero $x\in K_{\emptyset}$, then at least one of the
linear factors $(st-\lambda_i)$ must have nontrivial kernel,
contradicting Theorem~\ref{thm:W-eigen}.
\end{proof}
 
Now we suppose $M$ is an $(m\times n)$-matrix
with $\bbR G$-entries.  We let $R_M:\LG^m\rightarrow\LG^n$ denote
right multiplication by $M$.  Then $\ker R_M$ is a left $\NG$-module,
hence, by Proposition~\ref{prop:V-G-decomp}, decomposes as 
\[\ker R_M=(\ker R_M)_+\oplus(\ker R_M)_-\oplus(\ker
R_M)_{\emptyset}.\]
Moreover, each summand can be regarded as the kernel of right
multiplication by $M$ on the corresponding invariant subspace of
$\LG^m=(K_+)^m\oplus (K_-)^m\oplus (K_{\emptyset})^m$.  More precisely, if
$R_M^+:(K_+)^m\rightarrow (K_+)^m$, $R_M^-:(K_-)^m\rightarrow (K_-)^m$,
and $R_M^{\emptyset}:(K_{\emptyset})^m\rightarrow (K_{\emptyset})^m$ each denotes right multiplication
by the matrix $M$, then 
\[(\ker R_M)_+=\ker R_M^+,\;\; (\ker
R_M)_-=\ker R_M^-,\;\;\mbox{and}\;  (\ker
R_M)_{\emptyset}=\ker R_M^{\emptyset}.\]

\begin{lemma}\label{lem:G-main}
Let $M$ be a matrix with $\bbR G$-entries, and let $R_M^+$, $R_M^-$,
and $R_M^{\emptyset}$ denote right multiplication by $M$ on $(K_+)^m$,
$(K_-)^m$, and $(K_{\emptyset})^m$, respectively.  Then there exist  $\NG$-module
isomorphisms 
\[\ker R_M^+  \cong (K_+)^{a},\;\;\;\;\ker R_M^-  \cong (K_-)^{b},\;\;\;\;\mbox{and}\;\;
\ker R_M^{\emptyset}  \cong (K_{\emptyset})^{c},\]
for some choice of integers
$a,b,c\in\{0,1,\ldots,m\}$. 
\end{lemma}

\begin{proof}{}
Adding a zero column to $M$ does not effect the kernel of $R_M^+$,
$R_M^-$, or $R_M^{\emptyset}$, and adding a zero row only alters the
kernel by a free summand of $K_+$, $K_-$, or $K_{\emptyset}$,
respectively.  We can therefore assume that $M$ is a square matrix of
size $m\times m$.  The entries of $M$ are elements of $\bbR G$, which
we regard as the ring of Laurent polynomials in $z=st$ over $\bbR$.
Since right multiplication by $z=st$ (a unitary operator on
$\LG^n$) defines an $\NG$-module automorphism of $(K_+)^m$, $(K_-)^m$, and
$(K_{\emptyset})^m$, resp., we can multiply $M$ by any power of
$z$ without changing the kernel of $R_M^+$, $R_M^-$, or
$R_M^{\emptyset}$, resp.  Thus, we can assume that $M$ has polynomial
entries. Since polynomials over $\bbR$ form a principal
ideal domain, we can multiply $M$ on the right and left by invertible
matrices (over $\bbR G$) to obtain a diagonal matrix.  Hence the proof
of the lemma reduces to the case where $M$ is a diagonal matrix
$\diag(y_1,\ldots,y_m)$.  Finally we simply recall, from
Proposition~\ref{prop:Virred} and the paragraph preceding it, 
that right multiplication on $K_+$, $K_-$, or $K_{\emptyset}$ by any
element $y_i\in\bbR G$ is either an isomorphism or the zero map.  The
result follows.
\end{proof}

Finally, we consider $\NW$-modules that are
kernels of $\bbR W$-matrices.  Let $M$ be an $(m\times n)$-matrix
with $\bbR W$-entries, and let $R_M:(\LW)^m\rightarrow(\LW)^n$ denote
right multiplication by $M$.  As in the case of $\bbR G$-matrices, we
obtain a decomposition of left $\NW$-modules:
\begin{align}\label{al:kerRM-decomp}
\ker R_M &=\ker R_M^+\oplus \ker R_M^-\oplus \ker
R_M^{\emptyset},
\end{align}
where $R_M^+:(K_+)^m\rightarrow (K_+)^m$, $R_M^-:(K_-)^m\rightarrow (K_-)^m$,
and $R_M^{\emptyset}:(K_{\emptyset}\oplus K_{\emptyset}s)^m\rightarrow (K_{\emptyset}\oplus K_{\emptyset}s)^m$ each
denotes right multiplication by the matrix $M$.  These three summands
are also left $\NG$-modules, however, in order to use
Lemma~\ref{lem:G-main} , we need to know that as $\NG$-modules they
are isomorphic to kernels of $\bbR G$-matrices.  

\begin{lemma}\label{lem:kerW2G}
Let $M$ be an $(m\times n)$-matrix with entries in $\bbR W$.  Then
there exist $(m\times n)$-matrices $M_+$ and $M_-$, and a
$(2m\times 2n)$-matrix $M_{\emptyset}$ all with entries in $\bbR G$ such
that as $\NG$-modules,
\[\ker R_M^+\cong \ker R_{M_+},\;\; \ker R_M^-\cong \ker R_{M_-},\;\;\mbox{and}\;\; 
\ker R_M^{\emptyset}\cong \ker R_{M_{\emptyset}},\]
where $R_{M_+}$ denotes right-multiplication by
$M_+$ on $(K_+)^m$, $R_{M_-}$ denotes right-multiplication by
$M_-$ on $(K_-)^m$, and $R_{M_{\emptyset}}$ denotes right-multiplication
by $M_{\emptyset}$ on $(K_{\emptyset})^{2m}$.  
\end{lemma}

\begin{proof}{}
Any element $y$ in $\bbR W$ can be written in the form
$y=y_1(z)+y_2(z)s$ where $y_1(z)$ and $y_2(z)$ are Laurent polynomials
in $z=st$.  Moreover, since $(st)^ns=s(ts)^n=s(st)^{-n}$, any Laurent
polynomial $f(z)\in\bbR G$ satisfies the relation $f(z)s=sf(z^{-1})$
in $\bbR W$.  These same properties hold for any matrix $M$ with $\bbR
W$ entries.  Given such a matrix $M$, we let $M=M_1(z)+M_2(z)s$ where
$M_1(z)$ and $M_2(z)$ are $(m\times n$)-matrices with entries in $\bbR
G$.  Given $x\in (K_+)^m$, we have $x=x\tk_+$, so 
\begin{align*}
xM & =x\tk_+(M_1(z)+M_2(z)s)\\
&=x\tk_+M_1(z)+x\tk_+sM_2(z^{-1})\\
&=x\tk_+M_1(z)\pm x\tk_+M_2(z^{-1})\hspace{.3in}\mbox{(sign depending
on $\bq$)}\\
&=x\tk_+(M_1(z)\pm M_2(z^{-1}))\\
&=x(M_1(z)\pm M_2(z^{-1})).
\end{align*}
In other words, right multiplication by $M$ on $(K_+)^m$ is the same as
right multiplication by $M_1(z)\pm M_2(z^{-1}))$, which has entries in
$\bbR G$.  Letting $M_+$ be the matrix $M_+= M_1(z)\pm M_2(z^{-1}))$,
we therefore have $\ker R_M^+\cong \ker R_{M_+}$, as desired.
A similar argument works for $R_M^-$ acting on $(K_-)^m$.  

For $x\in(K_{\emptyset}\oplus K_{\emptyset}s)^m$, we express it as
$x=x_1+x_2s$ where $x_1,x_2\in (K_{\emptyset})^m$.  Then 
\begin{align*}
xM &= (x_1+x_2s)(M_1(z)+M_2(z)s)\\
&=x_1(M_1(z)+M_2(z)s)+x_2s(M_1(z)+M_2(z)s)\\
&= x_1M_1(z)+x_1M_2(z)s+x_2M_1(z^{-1})s+x_2M_2(z^{-1})\\
&= [x_1M_1(z)+x_2M_2(z^{-1})]+[x_1M_2(z)+x_2M_1(z^{-1})]s.
\end{align*}
It follows that if we identify $(K_{\emptyset}\oplus
K_{\emptyset}s)^m$ with $(K_{\emptyset})^m\oplus (K_{\emptyset})^m$ (using
the $\NG$-isomorphism $x_1+x_2s\mapsto (x_1,x_2)$), then right
multiplication by $M$ corresponds to right multiplication by the
$(2m\times 2n)$ block matrix
\[M_{\emptyset}=\left[\begin{array}{cc}
M_1(z) & M_2(z)\\
M_2(z^{-1}) & M_1(z^{-1})\end{array}\right].\]
Hence the two matrices $M$ and $M_{\emptyset}$ will have isomorphic
kernels (as $\NG$-modules).
\end{proof}

We now prove the main theorem of the paper.

\begin{theorem}\label{thm:final}
If $M$ is any $(m\times n)$-matrix with entries in $\bbR W$ and
 $R_M:(\LW)^m\rightarrow(\LW)^n$ denotes right multiplication by $M$, then  
\[\dim_W^{\bq}\ker R_M=n_{\emptyset}+\frac{n_s}{1+q_s}+\frac{n_t}{1+q_t}\]
where $n_{\emptyset},n_s,n_t$ are piecewise constant integer functions
of $q_s$ and $q_t$ with jumps only along the curves $q_s=q_t$ and
$q_sq_t=1$.
\end{theorem} 

\begin{proof}{}
By (\ref{al:kerRM-decomp}), we have 
\[\dim_W^{\bq}\ker R_M=\dim_W^{\bq}\ker R_M^++\dim_W^{\bq}
\ker R_M^-+\dim_W^{\bq} \ker R_M^{\emptyset},\]
hence by Lemma~\ref{lem:dimG2W}, we have 
\begin{align}\label{al:dimker}
\dim_W^{\bq}\ker R_M & =\dim_G^{\bq} \ker R_M^++\dim_G^{\bq}
\ker R_M^-+\frac{1}{2}\dim_G^{\bq} \ker R_M^{\emptyset}.
\end{align}
By Lemma~\ref{lem:kerW2G}, all of these $\NG$-modules are isomorphic
to kernels of $\bbR G$-matrices, hence by Lemma~\ref{lem:G-main}, we have 
\begin{align}
\nonumber\dim_G^{\bq}\ker R_M^+ & =\dim_G^{\bq} (K_+)^a ,\\
\label{al:dimabc}\dim_G^{\bq}\ker R_M^- & =\dim_G^{\bq} (K_-)^b,\\
\nonumber\dim_G^{\bq}\ker R_M^{\emptyset} & =\dim_G^{\bq} (K_{\emptyset})^c
\end{align}
for some integers $a,b,c$.  Note that these integers are constant with
respect to the parameter $\bq$.  Combining (\ref{al:dimker}) and 
(\ref{al:dimabc}) we have  
\[\dim_W^{\bq}\ker R_M=a\cdot\dim^{\bq}_G K_+ +b\cdot\dim^{\bq}_G K_-
+\frac{c}{2}\cdot\dim^{\bq}_G K_{\emptyset},\]
and the theorem then follows from Lemma~\ref{lem:Gdims}.
\end{proof}

\bibliographystyle{amsplain}

\providecommand{\bysame}{\leavevmode\hbox to3em{\hrulefill}\thinspace}
\providecommand{\MR}{\relax\ifhmode\unskip\space\fi MR }
\providecommand{\MRhref}[2]{%
  \href{http://www.ams.org/mathscinet-getitem?mr=#1}{#2}
}
\providecommand{\href}[2]{#2}

\end{document}